\documentclass[12pt]{amsart}

\usepackage{amssymb}
\usepackage{amsthm}
\usepackage{mathtools}
\usepackage{tikz-cd}
\usepackage{color}
\usepackage{hyperref}
\usepackage{upref}
\usepackage{enumitem}
\usepackage{verbatim}

\usepackage[latin1]{inputenc}

\setlist[enumerate]{nosep}
\setlist[enumerate,1]{label={\textup{(\arabic*)}}}

\allowdisplaybreaks

\newtheorem{thm}{Theorem}[section]
\newtheorem{prop}[thm]{Proposition}

\theoremstyle{definition}

\newtheorem{rem}[thm]{Remark}

\numberwithin{equation}{section}

\newcommand{\secref}[1]{Section~\textup{\ref{#1}}}
\newcommand{\apxref}[1]{Appendix~\textup{\ref{#1}}}
\newcommand{\thmref}[1]{Theorem~\textup{\ref{#1}}}

\newcommand{\propref}[1]{Proposition~\textup{\ref{#1}}}

\newcommand{\diagref}[1]{Diagram~\textup{\ref{#1}}}

\newcommand{\cc}[1]{\mathcal{#1}}

\newcommand{\KK}{\cc K}

\newcommand{\FF}{\cc F}

\newcommand{\CC}{\cc C}

\renewcommand{\epsilon}{\varepsilon}

\renewcommand{\:}{\colon}
\newcommand{\inv}{^{-1}}
\newcommand{\wilde}{\widetilde}
\newcommand{\what}{\widehat}

\newcommand{\id}{\text{id}}

\renewcommand{\subset}{\subseteq}

\renewcommand{\bar}{\overline}

\newcommand{\variso}{\overset{\simeq}{\longrightarrow}}
\newcommand{\xt}{\otimes}

\newcommand{\xts}{\xt_*\id}

\renewcommand{\)}{\textup)}
\newcommand{\rt}{\textup{rt}}
\newcommand{\lt}{\textup{lt}}

\newcommand{\cst}{\ensuremath{C^*}-algebra}
\newcommand{\cstg}{\ensuremath{C^*(G)}}

\newcommand{\clspn}{\operatorname{\bar{span}}}
\newcommand{\ad}{\operatorname{Ad}}

\newcommand{\cs}{\ensuremath{\mathbf{C}^*}}
\newcommand{\ac}{\ensuremath{\mathbf{Ac}}}
\newcommand{\co}{\ensuremath{\mathbf{Co}}}
\newcommand{\eac}{\ensuremath{\mathbf{EAc}}}
\newcommand{\kco}{\ensuremath{\mathbf{KCo}}}
\newcommand{\cpc}{\ensuremath{\mathbf{CPC}}}
\newcommand{\cpa}{\ensuremath{\mathbf{CPA}}}
\newcommand{\com}{\ensuremath{\mathbf{Com}}}
\newcommand{\st}{\ensuremath{\mathbf{St}}}
\newcommand{\ma}{\ensuremath{\mathbf{Max}}}

\newcommand{\kalg}{$\KK$-decorated algebra}

\newcommand{\kfixco}{$\KK$-fixing coaction}

\newcommand{\midtext}[1]{\quad\text{#1}\quad}
\newcommand{\righttext}[1]{\quad\text{#1 }}

\begin{document}

\title
{
Coactions on $C^*$-algebras and universal properties
}
\begin{abstract}
{
It is well-known that the maximalization of a coaction of a locally compact group on a C*-algebra enjoys a universal property. We show how this important property can be deduced from a categorical framework by exploiting certain properties of the maximalization functor for coactions. We also provide a dual proof for the universal property of normalization of coactions.
} 
\end{abstract}
\author[B\'edos]{Erik B\'edos}
\address{Department of Mathematics, University of Oslo, PB 1053 Blindern, 0316 Oslo, Norway}
\email{bedos@math.uio.no}
\author[Kaliszewski]{S. Kaliszewski}
\address{School of Mathematical and Statistical Sciences, Arizona State University, Tempe, AZ 85287}
\email{kaliszewski@asu.edu}
\author[Quigg]{John Quigg}
\address{School of Mathematical and Statistical Sciences, Arizona State University, Tempe, AZ 85287}
\email{quigg@asu.edu}
\author[Turk]{Jonathan Turk}
\address{School of Mathematical and Statistical Sciences, Arizona State University, Tempe, AZ 85287}
\email{jturk2@asu.edu}
\date{August 15, 2023}

\thanks{The first author is grateful to the Trond Mohn Foundation for financial support through the project ``Pure Mathematics in Norway'' making it possible for him to make a stay at the Arizona State University (Tempe) in April 2022.}


\maketitle

\section{Introduction}\label{intro}

Crossed-product duality for $C^*$-algebraic dynamical 
systems requires coactions as well as actions,
and coactions of a locally compact group can come in various flavors when the group is non-amenable.
The two most fundamental flavors are the extremes:
at the bottom are the \emph{normal coactions}, and at the top are the \emph{maximal coactions}.
When the duality theory is cast in categorical terms,
various functors appear, in addition to the ones coming from the obvious crossed-product constructions.
One of these is the \emph{maximalization functor} \cite{dualitiescoactions}, which goes back to the maximalization process introduced 
and studied by Fischer in \cite{fischer}, as a follow-up of \cite{maximal}.

This paper originated when, a few years ago, some of us got puzzled by the final one-line argument in the proof of one of Fischer's foundational theorems, 
 stating that a maximalization of a coaction enjoys a certain universal property.
Having a look at our alternative proof of this result in \cite{reflective} didn't give us any relief as we 
 soon realized that our proof also ought to be supplied with some additional explanation. 
Happily, we have very recently found a new approach to fill this perceived gap, which we present in the current paper.
We formulated the technical result and its proof in categorical terms, since we believe that this is the clearest way to see the ``foundational bedrock'' of the theorem.

More specifically, Fischer's theorem asserts that a certain diagram has a unique completion. Fischer's proof culminates in a much larger diagram,
that includes the diagram of interest, and ends by showing that this larger diagram has a unique completion,
and asserting that this diagram implies the result (although even this assertion is presented in a slightly cryptic form).
We have not communicated with Fischer concerning this, and of course we cannot conclude that Fischer ``forgot'' to relate the larger diagram to the smaller, important, one.
However, due to the extensive influence that this universal property has had on the subsequent development of the theory of cross-product duality, we feel compelled to make all this as clear as possible.

We begin in \secref{categories} with a lightning review of the relevant facts concerning actions and coactions, with (almost) no proofs, 
our main aim being to summarize the functorial approach to the maximalization process for coactions described in \cite{dualitiescoactions}.
But along the way we also introduce various categories and functors that we will need; this gives the review a somewhat nonstandard flavor.
The appendix of \cite{enchilada} has most of the details regarding actions and  normal coactions (for maximal coactions see \cite{maximal}), 
but (perhaps surprisingly) with not such a sharp focus on category theory.  Our presentation relies otherwise on \cite{fullred, maximal, stable, dualitiescoactions}. 
We must also mention that Fischer's article \cite{fischer} is in some sense our main source.

In \secref{gap},
we briefly sketch Fischer's original proof of his universal property of maximalization,
in order to pin-point where we feel this proof suffers from a lack of explanation.
We also mention that our slightly different proof appearing in \cite{reflective} has a similar gap.

In \secref{abstract},
we give an abstract universal property (\propref{aup}) in the setting of any category together with a functor and a natural transformation satisfying a quite simple axiom.
Then we indicate how this abstract result can be used to quickly prove Fischer's universal property for maximalization.

In \secref{dualizing},
we give an abstract result that is dual to \propref{aup},
and show in \thmref{univ-normal} how the universal property of normalization of coactions follows.

In \apxref{stab eq},
we fill a technical gap in the literature, concerning ``destabilization''.
If one allows ``stabilization'' of a \cst\ to mean tensoring with the compact operators $\KK$ on an arbitrary Hilbert space,
in order to recognize an algebra of the form $A\xt\KK$,
we need a completely general result, \propref{stable}.
Various authors have recorded a proof in the separable case (see \cite[Proposition~3.4]{stable}, for example),
but we could not find
a result in the literature at the level of generality we need.
Fortunately, the existing argument can be routinely modified.

We thank the anonymous referees for comments that significantly improved our paper.

\section{The categories and functors}\label{categories}

Throughout, $G$ is an arbitrary locally compact group (with a fixed left Haar measure), and $A,B,C,\dots$ are $C^*$-algebras.
We use $\otimes$ to denote the minimal tensor product of \cst s, and $M(A)$ to denote the multiplier algebra of $A$.

There is a category \cs\ of $C^*$-algebras, in which the morphisms
$\phi\:A\to B$ are nondegenerate homomorphisms $\phi\:A\to M(B)$.
Morphisms extend uniquely to multiplier algebras of their domain, and we frequently keep the same notation for the extension.
A morphism $u\:C^*(G)\to B$ is the integrated form
of a strictly continuous unitary homomorphism $G\to M(B)$, and we use the same notation for both.
This is the basis for all our categories: $C^*$-algebras decorated with extra structure, which the morphisms must preserve.

\subsection*{Actions}

There is a category \ac\ of actions $(A,\alpha)$ of $G$, in which morphisms are \cs-morphisms that
are equivariant
for the actions.
A morphism $u\:C^*(G)\to B$ in $\cs$ gives rise to a \emph{unitary action} $\ad u$,
and if $\pi\:(A,\alpha)\to (B,\ad u)$ is a morphism in \ac,
then $(B,\pi,u)$ is a \emph{covariant representation} of $(A,\alpha)$.
A \emph{\(full\) crossed product} of $(A,\alpha)$ is a universal covariant representation $(A\rtimes_\alpha G,i_A,i_G)$,
i.e., for every covariant representation $(B,\pi,u)$ of $(A,\alpha)$ there is a unique morphism $\pi\times u\:A\rtimes_\alpha G\to B$ in \cs\ such that
\[
(\pi\times u)\circ i_A=\pi
\righttext{and}
(\pi\times u)\circ i_G=u.
\]
By abstract nonsense, any two crossed products of $(A,\alpha)$ are 
unique up to unique isomorphism.
We (imagine that we) pick one and call it ``the'' crossed product.

Throughout this paper\footnote{Except in \propref{stable}, where $\KK=\KK(H)$ for an arbitrary Hilbert space~$H$.}, we write
$\KK=\KK(L^2(G))$
for the \cst\ of compact operators on $L^2(G)$ (defined with respect to the chosen left Haar measure), 
$M:C_0(G)\to \KK$ for the nondegenerate representation of $C_0(G)$ by multiplication operators, $\rho$ and $\lambda$ for the right and left 
regular representations of $G$ on $L^2(G)$, respectively,
and $\rt$ and $\lt$ for the actions of $G$ on $C_0(G)$ by right and left translation, respectively.
We use without comment the Stone-von Neumann Theorem: $(\KK,M,\rho)$  is a crossed product of the action $(C_0(G),\rt)$.

\subsection*{Coactions}

The theory of coactions is dual to that of actions,
and when $G$ is abelian coactions correspond via the Fourier transform to actions of the dual group $\what G$.
Technically, a (full) \emph{coaction} of $G$ on $A$ is a pair $(A,\delta)$,
where $\delta\:A\to A\xt C^*(G)$ is a morphism in \cs,
i.e., $\delta$ is a nondegenerate homomorphism into $M(A\xt C^*(G))$,
such that
\[
\clspn\{\delta(A)(1\xt \cstg)\}=A\xt \cstg
\]
and
the diagram
\[
\begin{tikzcd}
A \arrow[r,"\delta"] \arrow[d,"\delta",swap]
&A\xt C^*(G) \arrow[d,"\delta\xt\id"]
\\
A\xt C^*(G) \arrow[r,"\id\xt\delta_G"]
&A\xt C^*(G)\xt C^*(G)
\end{tikzcd}
\]
commutes in $\cs$.
Here the morphism $\delta_G\:C^*(G)\to C^*(G)\xt C^*(G)$ is the integrated form of the canonical unitary homomorphism $s\mapsto s\xt s$ for $s\in G$.
There is a category $\co$ of coactions of $G$, in which the morphisms $\phi:(A,\delta)\to (B,\epsilon)$
are \cs-morphisms $\phi:A\to B$ that are $\delta-\epsilon$ equivariant, i.e., the diagram
\[
\begin{tikzcd}
A 
\arrow[r,"\delta"] \arrow[d,"\phi",swap]
&A\xt C^*(G)
\arrow[d,"\phi\xt\id"]
\\
B 
\arrow[r,"\epsilon",swap]
&B\xt C^*(G)
\end{tikzcd}
\]
commutes.
A morphism $\mu\:C_0(G)\to B$ in $\cs$ gives rise to a \emph{unitary coaction}
$\delta_\mu$ on $B$
defined by the formula
\[
\delta_\mu(b)=\ad (\mu\xt\id(w_G))(b\xt 1),
\]
and if  $\pi\:(A,\delta)\to (B,\delta_\mu)$ is a morphism in \co,
then $(B,\pi,\mu)$ is a \emph{covariant representation} of $(A,\delta)$.
A (full) \emph{crossed product} of $(A,\delta)$ is a universal covariant representation
$(A\rtimes_\delta G,j_A,j_G)$,
i.e., for every covariant representation $(B,\pi,\mu)$ of $(A,\delta)$ there is a unique morphism $\pi\times \mu\:A\rtimes_\delta G\to B$ in \cs\ such that
\[
(\pi\times \mu)\circ j_A=\pi
\righttext{and}
(\pi\times \mu)\circ j_G=\mu.
\]
By abstract nonsense, any two crossed products of $(A,\delta)$ are 
unique up to unique isomorphism,
and we 
pick one and call it ``the'' crossed product.

If $(A,\delta)$ is a coaction, then there is a coaction on $A\xt\KK$ given by
\[
\delta\xts:=(\id\xt\Sigma)\circ(\delta\xt\id),
\]
where
\[
\Sigma\:C^*(G)\xt\KK\simeq \KK\xt C^*(G)
\]
is the flip isomorphism determined by $c\xt k\mapsto k\xt c$, and there is a natural isomorphism
\[
(A\xt\KK)\rtimes_{\delta\xts} G\simeq (A\rtimes_\delta G)\xt \KK.
\]

If $(A,\alpha)$ is an action, there is a canonical coaction $\what\alpha$ on $A\rtimes_\alpha G$,
called the \emph{dual coaction},
determined by
\begin{align*}
\what\alpha\circ i_A&=i_A\xt 1\\
\what\alpha\circ i_G(s)&=i_G(s)\xt s,
\end{align*}
and \emph{Imai-Takai duality}
{(see \cite[Theorem~A.67]{enchilada} for this formulation)}
says that
\[
A\rtimes_\alpha G\rtimes_{\what\alpha} G\simeq A\xt \KK.
\]
Similarly, if $(A,\delta)$ is a coaction, there is a \emph{dual action} $\what\delta$ on $A\rtimes_\delta G$,
where for $s\in G$ the automorphism $\what\delta_s$ is determined by
\begin{align*}
\what\delta_s\circ j_A&=j_A\\
\what\delta_s\circ j_G&=j_G\circ \rt_s.
\end{align*}
There is a \emph{canonical surjection}
{(see \cite[Corollary~2.6]{nilsen})}
\[
\Phi_{(A, \delta)}\:A\rtimes_\delta G\rtimes_{\what\delta} G\to A\xt\KK,
\]
and $\delta$ is \emph{maximal} if $\Phi_{(A,\delta)}$ is an isomorphism
{(\cite[Definition~3.1]{maximal})};
this is called \emph{Katayama duality}
for maximal coactions.

If $\phi\:(A,\delta)\to (B,\epsilon)$ is a morphism in $\co$, then there is a 
canonical \emph{crossed-product} morphism
\[
\phi\rtimes G\:(A\rtimes_\delta G,\what\delta)\to (B\rtimes_\epsilon G,\what\epsilon)
\]
in \ac, and this induces a functor $\co\to\ac$.
Similarly, if $\phi\:(A,\alpha)\to (B,\beta)$ is a morphism in \ac, then
there is a canonical \emph{crossed-product} morphism
\[
\phi\rtimes G\:(A\rtimes_\alpha G,\what\alpha)\to (B\rtimes_\beta G,\what\beta)
\]
in $\co$, and this gives a functor $\ac\to\co$. 

\label{max def}
A \emph{maximalization} of a coaction $(A,\delta)$ is a triple $(B,\epsilon,\psi)$ consisting of a maximal coaction $(B,\epsilon)$  
and a morphism $\psi:(B,\epsilon)\to (A,\delta)$ in \co\ such that $\psi\rtimes G:B\rtimes_\epsilon G\to A\rtimes_\delta G$ is an isomorphism of \cst s.

\subsection*{$\KK$-decorated algebras}

A \emph{\kalg} is a pair $(A,\iota)$, where $\iota:\KK\to A$ is a morphism in \cs.
There is a category of \kalg s in which a morphism $\phi\:(A,\iota)\to (B,\jmath)$ is a \cs-morphism $\phi\:A\to B$ such that
$\phi\circ\iota=\jmath$. We occasionally refer to the map $\iota$ as a \emph{$\KK$-decoration}.

If $(A,\delta)$ is a coaction and $(A,\iota)$ is a \kalg, 
then $\delta$ is \emph{$\KK$-fixing} if $\delta\circ\iota=\iota\xt 1$, 
and we then call $(A,\delta,\iota)$ a \emph{\kfixco}.  
The categories of coactions and of \kalg s combine immediately to form a 
category \kco\ of \kfixco s in which a morphism $\phi\:(A,\delta, \iota)\to (B, \epsilon, \jmath)$ 
 is a morphism $\phi\:(A, \iota) \to (B, \jmath)$ of $\KK$-decorated algebras that is $\delta-\epsilon$ equivariant.

The \emph{relative commutant} of a \kalg\ $(A,\iota)$ is the \cst
\[
C(A,\iota)=\{m\in M(A):m\iota(k)=\iota(k)m\in A\text{ for all }k\in\KK\}.
\]

Our reason for introducing relative commutants is
the following underappreciated fact:
if $(A, \iota)$ is a \kalg, then there is an isomorphism
$C(A, \iota) \xt \KK\variso A$ given on elementary tensors by
$m \xt k\mapsto m\iota(k)$.
This is probably folklore;
a proof can be found in 
\cite[Proposition 3.4]{stable},
where we give
a more detailed alternative to the argument in \cite[Remark 3.1]{fischer} (see also \cite[Lemma~27.2]{exelbook}, \cite[Theorem~2.1]{HR}). As $G$ is assumed to be second-countable in \cite{stable, exelbook} (but not in \cite{fischer}), 
we have included a proof of the general case in
\propref{stable}
in order to get rid of this unnecessary assumption.

It follows that $C(A,\iota)$ is a nondegenerate $C^*$-subalgebra of $M(A)$,
and \cite[Remark~3.1]{fischer} characterizes it as
the unique closed subset $Z$ of $M(A)$
that commutes elementwise with $\iota(\KK)$ and satisfies
$\clspn\{Z\iota(\KK)\}=A$.
Moreover, Fischer further shows that $M(C(A,\iota))$ is the commutant of $\iota(\KK)$ in $M(A)$.
Using these facts, it is easy to show that
any morphism $\phi:(A,\iota)\to (B,\jmath)$ of \kalg s 
restricts (after being extended to $M(A)$) to a morphism $C(\phi):C(A,\iota)\to C(B,\jmath)$ in \cs.
If $(A,\delta,\iota)$ is a \kfixco, then
\[
\delta\:(A,\iota)\to (A\xt C^*(G),\iota\xt 1)
\]
is a morphism of \kalg s.
Taking relative commutants gives a morphism
\[
C(\delta):C(A,\iota)\to C(A\xt\cstg,\iota\xt 1),
\]
and Fischer's characterization quoted above
easily yields the equality
\[
C(A\xt C^*(G),\iota\xt 1) = C(A,\iota)\xt C^*(G),
\]
from which we conclude
(see \cite[Remark~3.2]{fischer}, and also \cite[Lemma~3.2]{dualitiescoactions})
that $C(\delta)$ is a coaction on $C(A,\iota)$.
Moreover, if $\phi: (A,\delta, \iota)\to (B, \epsilon, \jmath)$ is a morphism in \kco\ then $C(\phi):(C(A,\iota), C(\delta)) \to (C(B,\jmath), C(\epsilon))$ gives a morphism in \co.
In this way the relative commutant induces a functor $\com:\kco\to \co$,
called the nondegenerate destabilization functor in \cite{stable}, given by the assignments
\[  (
A,\delta,\iota)\mapsto (C(A,\iota),C(\delta))
\midtext{and}
\phi\mapsto C(\phi).
\]
In fact, \com\ is a category equivalence, with quasi-inverse $\st:\co\to\kco$ given by the assignments
\begin{align*}
    (A,\delta)&\mapsto (A\xt \KK,\delta\xts,1\xt \id_\KK)\\
    \phi&\mapsto \phi\xt\id_\KK.
\end{align*}
Equivalently, the stabilization functor \st\ is a category equivalence with quasi-inverse \com. 
%
For each $\KK$-fixing coaction $(A,\delta,\iota)$, the isomorphism $C(A,\iota)\xt\KK\variso A$ described above is $(C(\delta)\xt\id)-\delta$ equivariant,
and we denote it by $\theta_{(A,\delta,\iota)}$.
These isomorphisms combine to give a
natural isomorphism $\theta:\st\circ\com\variso \id_{\kco}$.
Moreover, category theory tells us that there is also a unique natural isomorphism
$\tau:\id_\co\variso \com\circ\st$ satisfying the \emph{triangle identities}
\[
\begin{tikzcd}
\st \arrow[r,"\st\,\circ\,\tau"] \arrow[dr,"1_{\st}"']
&\st\circ\com\circ\st \arrow[d,"\theta\,\circ\,\st"]
\\
&\st
\end{tikzcd}
\midtext{and}
\begin{tikzcd}
\com \arrow[r,"\tau\,\circ\,\st"] \arrow[dr,"1_{\com}"']
&\com\circ\st\circ\com \arrow[d,"\com\circ\theta"]
\\
&\com
\end{tikzcd}
\]
More precisely, $\theta$ is the counit for the adjunction $\st\dashv\com$, and then $
\tau$ is the unique unit (for the uniqueness, see for example, \cite[Theorem~2.3.6]{leinster}, or more accurately its dual version).
Of course, the unit is a natural isomorphism since the counit is.

Actually, it will be more convenient for us to have the inverse natural isomorphism
\begin{equation}\label{kappa}
\kappa:=\tau\inv:\com\circ\st\variso \id_\co.
\end{equation}

\subsection*{Cocycles}

If $(A,\delta)$ is a coaction, a \emph{$\delta$-cocycle} is a unitary $u\in M(A\xt C^*(G))$
such that
\begin{itemize}
\item
$\id\xt\delta_G(u)=(u\xt 1)(\delta\xt\id)(u)$ and
\item
$u\delta(A)u^*(1\xt \cstg)\subset A\xt \cstg$.
\end{itemize}
The above axioms are designed so that
$\delta^u:=\ad u\circ\delta$ is another coaction on $A$,
called a \emph{perturbation} of $\delta$ by $u$ or a \emph{perturbed coaction}.
Cocycles are natural in the following sense:
if $\phi\:(A,\delta)\to (B,\epsilon)$ is a morphism in $\co$ and $u$ is a $\delta$-cocycle,
then
$(\phi\xt\id)(u)$
is an $\epsilon$-cocycle, and $\phi$ is also 
$\delta^u-\epsilon^{(\phi\xt\id)(u)}$
equivariant.

Moreover, there is an isomorphism
\[
\Omega_u\:(A\rtimes_{\delta^u} G,\what{\delta^u})
\variso (A\rtimes_\delta G,\what\delta)
\]
in \ac\ for each $\delta$-cocycle $u$, and furthermore this family of isomorphisms is
natural in the sense that if $\phi\:(A,\delta)\to (B,\epsilon)$ is a morphism in \co\ then the diagram
\[
\begin{tikzcd}
\bigl(A\rtimes_{\delta^u} G,\what{\delta^u}\bigr)
\arrow[r,"\phi\rtimes G"] \arrow[d,"\Omega_u"',"\simeq"]
&\bigl(B\rtimes_{\epsilon^{(\phi\xt\id)(u)}} G,\what{\epsilon^{(\phi\xt\id)(u)}}\bigr)
\arrow[d,"\Omega_{(\phi\xt\id)(u)}","\simeq"']
\\
(A\rtimes_\delta G,\what\delta)
\arrow[r,"\phi\rtimes G",swap]
&(B\rtimes_\epsilon G,\what\epsilon)
\end{tikzcd}
\]
commutes in \ac.

An \emph{equivariant action} is a triple $(B,\alpha,\mu)$, where $(B, \alpha)$ is an action and $\mu:(C_0(G),\rt)\to (B,\alpha)$ is a morphism in \ac.
The category \eac\ has equivariant actions $(B,\alpha,\mu)$ as objects,
and a morphism $\phi:(B,\alpha,\mu)\to (C,\beta,\nu)$ in \eac\ is an \ac-morphism $\phi:(B,\alpha)\to (C,\beta)$
such that $\phi\circ\mu=\nu$.


For any equivariant action $(B,\alpha,\mu)$,
by \cite[Lemma~3.1]{dualitiescoactions}
the unitary $u:=(i_B\circ\mu\xt\id)(w_G)$
is
an $\what\alpha$-cocycle 
such that the perturbed coaction
$\wilde\alpha:=(\what\alpha)^u$ is $\KK$-fixing, where we use the $\KK$-decoration
\[
\mu\rtimes G:\KK=C_0(G)\rtimes_{\rt} G\to B\rtimes_\alpha G.
\]
Thus, the triple $(B\rtimes_\alpha G,\wilde\alpha,\mu\rtimes G)$ is a \kfixco, and the crossed product induces a functor $\cpa:\eac\to \kco$ 
with assignments
\[
(B,\alpha,\mu)\mapsto (B\rtimes_\alpha G,\wilde\alpha,
\mu\rtimes G)
\midtext{and}
\phi\mapsto \phi\rtimes G.
\]

On the other hand, for any coaction $(A,\delta)$, the triple $(A\rtimes_\delta G,\what\delta,j_G)$
is an equivariant action, and the crossed product induces a functor $\cpc:\co\to\eac$ with assignments
\[
(A,\delta)\mapsto (A\rtimes_\delta G,\what\delta,j_G)
\midtext{and}
\phi\mapsto \phi\rtimes G.
\]
Then we can apply the above to obtain a \kfixco\
\[
(A\rtimes_\delta G\rtimes_{\what\delta} G,\wilde\delta,j_G\rtimes G),
\]
where we have simplified the notation by
writing $\wilde\delta$ rather than $\wilde{\what\delta}$.
We compose to
get the functor $\cpa \circ \cpc:\co \to \kco$ with assignments
\[(A,\delta)\mapsto (A\rtimes_\delta G\rtimes_{\what\delta} G,\wilde \delta,j_G\rtimes G)
\midtext{and}
\phi\mapsto (\phi\rtimes G)\rtimes G.
\]

\subsection*{Maximalization}

Let $(A, \delta)$ be a coaction. The canonical surjection $\Phi_{(A,\delta)}:A\rtimes_\delta G\rtimes_{\what\delta} G\to A\otimes\KK$
is compatible with enough extra structure that it gives a morphism 
\[ 
\Phi_{(A,\delta)}:(A\rtimes_\delta G\rtimes_{\what\delta} G, \wilde \delta,j_G\rtimes G )\to (A\otimes\KK,\delta\xts,1\xt \id_\KK)
\]
in \kco. Moreover, the family of morphisms $\{\Phi_{(A, \delta)}\}$ gives a natural transformation
$\Phi:\cpa\circ\cpc\to \st$.

We define a functor $\ma:\co\to\co$ as the composition
\[
\ma:=\com\circ\cpa\circ\cpc.
\]
We write
\[
(A^m,\delta^m) := \ma(A,\delta)
\]
on objects in \co, and $\phi^m := \ma(\phi)$ on morphisms in \co.
Fischer proves in \cite[Theorem~6.4]{fischer} that the coaction $(A^m,\delta^m)$ is maximal.

Since the functor $\com:\kco\to \co$ is a category equivalence, it follows quickly that
$(A,\delta)$ is maximal if and only if $\com(\Phi_{(A,\delta)})$
is an isomorphism, which in turn is equivalent with
\[
\kappa_{\com(\Phi_{(A,\delta)})}:(A^m,\delta^m)\to (A,\delta)
\]
being an isomorphism, where $\kappa:\com\circ\st:\variso \id_\co$ is the natural isomorphism from \eqref{kappa}.

Thus $\psi:=\kappa\circ\com\circ\Phi$ gives a natural transformation $\ma\to\id_\co$,
and $(A,\delta)$ is maximal if and only if $\psi_{(A,\delta)}:(A^m,\delta^m)\to (A,\delta)$ is an isomorphism.
For any coaction $(A,\delta)$, Fischer (further) proves in \cite[Theorem~6.4]{fischer} that $(A^m,\delta^m,\psi_{(A,\delta)})$ is a maximalization of $(A,\delta)$.
It follows almost trivially that  $\cpc\circ \psi:\cpc\circ \ma\to \cpc$ is a natural isomorphism.
Then composing on the left with the functor  $\com\circ \cpa$,
 we deduce that
\[
\ma\circ \psi:\ma^2\to \ma
\]
is a natural isomorphism. In \secref{abstract} we will take these properties of maximalization as the starting point for a quite abstract proof of a
universal property.

\begin{rem}
In the literature the term ``maximalization'' has been used in several ways, and as a result there is potential for confusion.
For example, \cite[Definition~3.1]{maximal} and \cite[Definition~6.1]{fischer}
use the definition we quoted on page~\pageref{max def},
but \cite[Definition~6.1.3]{reflective} instead defines maximalization in terms of Fischer's universal property
\cite[Lemma~6.2]{fischer}.
We emphasize that in this paper we use the definition given on page~\pageref{max def}.
Then \cite[Theorem~3.3]{maximal} and \cite[Theorem~6.4]{fischer}
both prove that every coaction has a maximalization; note that there is no danger of circular reasoning in \cite{fischer}, because
careful examination of Fischer's proof reveals that it does not appeal to the universal property
\cite[Lemma~6.2]{fischer}.
Nevertheless, to avoid any confusion, 
we rely here only upon the statement and proof of
\cite[Theorem~3.3]{maximal}.
\end{rem}

\section{The gap}\label{gap}

In this short section we give a formal statement of Fischer's universal property,
and we briefly point out the gap in the proof.
Theorem \ref{refl-max} below reproduces Proposition~6.1.11 in \cite{reflective},
which includes
Fischer's universal property of maximalization
\cite[Lemma~6.2]{fischer}.


\begin{thm}[\cite{reflective}]\label{refl-max}
Let $(B,\epsilon)$ be a maximal coaction and $\psi\:(B,\epsilon)\to (A,\delta)$ be a morphism in \co. Then $(B,\epsilon,\psi)$ is a maximalization of $(A, \delta)$
if and only if for every maximal coaction $(C,\zeta)$ and every morphism $\tau\:(C,\zeta)\to (A,\delta)$ in \co\ there exists a unique morphism $\pi$ making the diagram
\begin{equation}\label{universal property}
\begin{tikzcd}
(C,\zeta) \arrow[d,"\pi"',"!",dashed] \arrow[dr,"\tau"]
\\
(B,\epsilon) \arrow[r,"\psi",swap]
&(A,\delta)
\end{tikzcd}
\end{equation}
commute in \co.
\end{thm}

One direction of \thmref{refl-max}
was originally stated in \cite{fischer}, but there is a gap in Fischer's argument for the uniqueness part.
The idea is to embed \diagref{universal property} in the following larger diagram:
\[
\begin{tikzcd}[column sep=huge,row sep=huge]
C\rtimes G\rtimes G\arrow[r,"\tau\rtimes G\rtimes G"] \arrow[d,
"\Phi_{(C, \zeta)}"',
"\simeq"]
&A\rtimes G\rtimes G \arrow[d,"\Phi_{(A,\delta)}",swap]
&B\rtimes G\rtimes G \arrow[l,
"\psi\rtimes G\rtimes G"',
"\simeq"] 
\arrow[d,
"\Phi_{(B, \epsilon)}",
"\simeq"']
\\
C\xt \KK \arrow[r,"\tau\xt\id",swap]
\arrow[rr,"\pi\xt\id"',"!",bend right,dashed]
&A\xt \KK
&B\xt \KK, \arrow[l,
"\psi\xt\id"]
\end{tikzcd}
\]
then use the various properties to deduce the existence of a unique morphism $\pi$ making the whole diagram commute.
However (upon recent reflection) 
it is unclear to us why that implies that $\pi$ is the unique morphism making
\diagref{universal property}
commute, a conclusion that
Fischer seems to take as obvious.
Another proof was given in \cite{reflective}, but there is again a (similar) gap with the uniqueness in that proof.

\section{The abstract approach}\label{abstract}

In this section we first prove an abstract universal property that assumes a couple of category-theoretic axioms. Then we apply this result to get a 
complete, short proof of the universal property of maximalization.

Suppose $\CC$ is a category, $M:\CC\to\CC$ is a functor, and $\psi:M\to \id_\CC$ is a natural transformation such that $M\psi:M^2\to M$
is a natural isomorphism. We recall that the natural transformation $M\psi$ is given by
\[
(M\psi)_x:=M\psi_x:M^2x\to Mx
\]
for each object $x$ of $\CC$.
Let $\CC_m$ be the full subcategory of $\CC$ formed by those objects $x$ for which $\psi_x$ is an isomorphism.

\begin{prop}[Abstract universal property]\label{aup}
If 
$f:y\to x$ in $\CC$ with $y\in\CC_m$, then there is a unique morphism 
$\wilde f$ in $\CC$ making the diagram
\[
\begin{tikzcd}
y \arrow[r,"\wilde f",dashed] \arrow[dr,"f"']
&Mx \arrow[d,"\psi_x"]
\\
&x
\end{tikzcd}
\]
commute.
\end{prop}

\begin{proof}
(Existence) The natural transformation $\psi$ gives the following commutative diagram:
\[
\begin{tikzcd}
My \arrow[r,"Mf"] \arrow[d,"\psi_y"']
&Mx \arrow[d,"\psi_x"]
\\
y \arrow[r,"f"']
&x
\end{tikzcd}
\]
Since $\psi$ is invertible, we can take $\wilde f=Mf\circ\psi_y\inv$

(Uniqueness) Suppose $h, k : y\to M x$ and $\psi_x h=\psi_x k=f$.
Applying the functor $M$ gives 
\[
M\psi_x\circ Mh=M\psi_x\circ Mk,
\]
and invertibility of $M\psi_x$ implies $Mh = Mk$. Applying the natural transformation $\psi$ to $h$ gives a commutative diagram
\[
\begin{tikzcd}
My \arrow[r,"Mh"] \arrow[d,"\psi_y"']
&M^2x \arrow[d,"\psi_{Mx}"]
\\
y \arrow[r,"h"']
&Mx
\end{tikzcd}
\]
Thus $h=\psi_{Mx}\circ Mh\circ\psi_y\inv$.
Similarly, we get $k=\psi_{Mx}\circ Mk\circ\psi_y\inv$, and hence
\[
k = \psi_{Mx}\circ Mh\circ\psi_y\inv = h.\qedhere
\]
\end{proof}

\subsection*{Applying to maximalization}

Let now
\begin{itemize}
\item
$\CC=\co$

\item
$M=\ma$

\item
$\psi:\ma\to\id_\CC$ 
\end{itemize}
be as in \secref{categories}. Then, as seen in this section, all the properties  of $\CC, M$ and $\psi$ 
assumed at the beginning of the present section are satisfied. Moreover,
 $\CC_m= \co_m$ is the full subcategory of \co\ whose objects are the maximal coactions.  
Applying \thmref{aup}, we immediately recover Fischer's universal property for maximalization, that is, Theorem \ref{aup} holds.

\section{Dualizing}\label{dualizing}

Our abstract result in Section \ref{abstract} has a dual version. For completeness we state this result below, and sketch how it can be used to show the universal property of normalization for coactions. 
 Since some direct proofs of this property are available (see \cite[Lemma 4.4]{hqrw} and \cite[Proposition 6.1.5 and Remark 6.1.8 ii)]{reflective}), our main purpose here is to illustrate that maximalization and normalization of coactions play a dual r\^ole.

 Suppose $\CC$ is a category, $N:\CC\to\CC$ is a functor, and $\eta:\id_\CC\to N$ is a natural transformation such that $N\eta:N\to N^2$ is a natural isomorphism. 
 
Let  $\CC_n$ be the full subcategory of $\CC$ formed by those objects $x$ for which $\eta_x$ is an isomorphism. Then we have:

\begin{prop}[Abstract universal property II]\label{aup2}
If  $g:x\to y$ in $\CC$ with $y\in\CC_n$, then there is a unique morphism 
$g' : Nx\to y$ in $\CC$  such that $g' \circ \eta_x = g$.
\end{prop}

The proof is similar to the proof of Proposition \ref{aup}. E.g., $g'$ is given by $g' = \eta_y\inv\circ Ng$.

\subsection*{Applying to normalization} 

Normality and normalization of a coaction can be defined in different (but equivalent) ways, see for instance 
\cite{fullred, maximal, enchilada, hqrw, reflective, dualitiescoactions}. We 
will follow \cite{fullred, enchilada}, but adapt the terminology to make it fit with the one used in \secref{categories}. 

 Let $(A\rtimes_\delta G, j_A, j_G)$ denote a crossed product for a given coaction $(A, \delta)$.
 We define $(A, \delta)$ to be \emph{normal} when $j_A: A \to M(A\rtimes_\delta G)$ is injective. (This does not depend on the choice of crossed product.)
 
 We recall from Proposition 2.3 in \cite{fullred} (or Lemma A.55 in \cite{enchilada}) that if $(B, \pi, \mu)$ is a covariant representation of $(A, \delta)$, 
 then there is a coaction,
which we denote here by $\delta_{(\pi,\mu)}$,
on $\pi(A)$ such that $(\pi(A), \delta_{(\pi,\mu)})$ is normal.  In particular, the coaction $(A^n, \delta^n):=(j_A(A), \delta_{(j_A,j_G)})$ is normal. 
 We call it the \emph{normalization of $(A, \delta)$} and denote by $\eta_{(A, \delta)}$ the homomorphism $j_A$, considered as a map from $A$ to $j_A(A)$. 
 (Alternatively, we could set $A^n= A/\ker j_A \simeq j_A(A)$ and let $\eta_{(A, \delta)}$ denote the quotient map.)    
   
   As $\eta_{(A, \delta)}$ is  nondegenerate and $\delta - \delta^n$ equivariant, $\eta_{(A, \delta)}$ is a morphism from $(A, \delta)$ to $(A^n, \delta^n)$ in $\co$. 
  For completeness we also mention that the induced  \cs-morphim 
  $\eta_{(A,\delta)} \rtimes G: A\rtimes _\delta G \to A^n\rtimes_{\delta^n} G$ 
  is an isomorphism that is $\widehat\delta - \widehat{\delta^n}$ equivariant, cf.~Lemma A.46 in \cite{enchilada} or Proposition 2.6 in \cite{fullred}. 
 
 The normalization functor on coactions is introduced in \cite{hqrw} and \cite{reflective}, but both these approaches build on the universal property 
 of normalization. We therefore indicate how this can be avoided. 
  Let $\phi : (A, \delta) \to (B, \varepsilon)$ be a morphism in $\co$. Then there is a unique morphism $\phi^n : (A^n, \delta^n) \to (B^n, \varepsilon^n)$ in $\co$ 
  such that 
 \begin{equation} \label{phi-n} \phi^n \circ \eta_{(A, \delta)} = \eta_{(B^n, \varepsilon^n)} \circ \phi.\end{equation}
 A proof of this statement  without using the universal property goes like this.\footnote{This argument is a special case of a more general argument 
 used in  the proof of \cite[Lemma 4.4]{hqrw}.}
 Since equation (\ref{phi-n}) says that $\phi^n(j_A(a)) = j_{B^n}(\phi(a))$ for all $a\in A$, the uniqueness part is clear. 
 This formula  will actually define $\phi^n$ if we know that $\ker j_A$ is contained in $\ker \phi$. Since $j_{B^n}$ is injective, it suffices to show 
 that $\ker j_A \subseteq \ker (j_{B^n}\circ \phi)$. As one readily verifies that $(j_{B^n}\circ\, \phi, j_G)$ is a covariant representation of $(A, \delta)$, 
 we get that $((j_{B^n}\circ \phi) \times j_G)\circ j_A = j_{B^n}\circ \phi$, so this assertion is clear. 
 It is then obvious  that $\phi^n$ is a \cs-morphism and it is straightforward to check that it is $\delta^n - \varepsilon^n$ equivariant. 
 
 We can now proceed and assert that  the assignments \[(A, \delta) \mapsto (A^n, \delta^n) \text{ and } \phi \mapsto \phi^n\] give a functor ${\rm{\bf Nor}}: \co \to \co$, 
 called the \emph{normalization functor}. To check that ${\rm{\bf Nor}}$ is really a functor is routine. For example, let $\phi : (A, \delta) \to (B, \varepsilon)$ and 
 $\psi :  (B, \varepsilon) \to (C, \gamma)$ be morphisms in $\co$. Then using (\ref{phi-n}), one gets that 
 $(\psi^n\circ \phi^n)\circ \eta_{(A, \delta)} = \eta_{(C^n, \gamma^n)}\circ (\psi\circ\phi)$. 
 By uniqueness, this gives that $(\psi\circ \phi)^n= \psi^n\circ \phi^n$, 
 i.e., ${\rm{\bf Nor}}(\psi\circ \phi) = {\rm{\bf Nor}}(\psi)\circ {\rm{\bf Nor}}(\phi)$.
 
We note  that equation (\ref{phi-n}) says that the  transformation $\eta: {\rm Id}_{\co} \to \rm{\bf Nor}$ associated to the family $\{\eta_{(A, \delta)}\}$ is natural.  It is also obvious that the full subcategory $\co_n$ of $\co$ whose objects $(A, \delta)$ are such that $\eta_{(A, \delta)}$ is an isomorphism is the category of normal coactions of $G$. Thus, to see that we may take  $N=\rm{\bf Nor}$ and $\CC = \co$ in Proposition \ref{aup2} to deduce that the universal property of normalization holds, it remains only to check that ${\rm{\bf Nor}}\, \eta:\rm{\bf Nor}\to \rm{\bf Nor}^2$ is a natural isomorphism. 

If  $(A, \delta)$ is any coaction, then ${\rm{\bf Nor}}\, \eta_{(A, \delta)} = (\eta_{(A, \delta)})^n$ is a morphism in $\co$ from $(A^n, \delta^n)$ to ${\rm{\bf Nor}} (A^n, \delta^n) = ((A^n)^n, (\delta^n)^n)$. Now,  $(\eta_{(A, \delta)})^n$ is uniquely determined by the identity
 \[(\eta_{(A, \delta)})^n \circ \eta_{(A,\delta)} = \eta_{(A^n, \delta^n)} \circ \eta_{(A,\delta)}.\]
 Since $\eta_{(A,\delta)}$ is surjective as a map from $A$ into $A^n$, this means that \[(\eta_{(A, \delta)})^n = \eta_{(A^n, \delta^n)}.\] 
 As $(A^n, \delta^n)$ is normal,  we know that $j_{A^n}$ is injective. Thus, we get that $(\eta_{(A, \delta)})^n = \eta_{(A^n, \delta^n)}: A^n \to (A^n)^n= j_{A^n}(A^n)$ is an isomorphism, as desired. 
 
 We have thereby given another proof of the universal property of normalization:
 
 \begin{thm}\label{univ-normal}
If $\phi\: (A,\delta)\to(B,\epsilon)$ is a morphism in \co\ and the coaction $(B,\epsilon)$ is normal,
then there is a unique morphism 
$\phi^n$
such that the diagram
\[
\begin{tikzcd}
(A,\delta) \arrow[d,"\eta_{(A,\delta)}"'] \arrow[dr,"\phi"]
\\
(A^n,\delta^n) \arrow[r,dashed, "\phi^n", swap]
&(B,\epsilon)
\end{tikzcd}
\]
commutes in \co.
\end{thm}
\appendix

\section{Stabilization equivalence} \label{stab eq}

In this appendix we record a completely general version of the equivalence between the categories \kco\ and \co\ mentioned in Section \ref{categories}; 
the separable case where $G$ is assumed to be second-countable appears in \cite[Theorem~4.4]{stable}.

\begin{thm}
The relative commutant functor $\com : \kco \to \co$ is a category equivalence, with quasi-inverse $\st:\co\to\kco$ that takes
an object $(A,\delta)$ to the $\KK$-fixing coaction $(A\xt\KK,\delta\xts,1\xt\id_\KK)$ and a morphism $\phi:(A,\delta)\to (B,\epsilon)$ to
$\phi\xt\id:(A\xt\KK,\delta\xts,1\xt\id_\KK)\to (B\xt\KK,\epsilon\xts,1\xt\id_\KK)$.
\end{thm}

\begin{proof}

The proof in \cite{stable} that $\st:\co\to\kco$ is a category equivalence, with quasi-inverse \com, goes through without any changes,
 except that we use \propref{stable} below instead of \cite[Proposition 3.4]{stable}.
\end{proof}


\begin{prop}\label{stable}
If $(A, \iota)$ is a \kalg, then there is an isomorphism
$\theta : C(A, \iota) \xt \KK\variso A$ given on elementary tensors by
$\theta(m \xt k) = m\iota(k)$.
\end{prop}

\begin{proof}
We outline how the argument from \cite[Proposition~3.4]{stable} can be adapted to cover the general case.
Choose a system $\{u_{ij}\}_{i,j\in S}$ of matrix units for $\KK$, and let $e_{ij}=\iota(u_{ij})$ be the corresponding matrix units in $M(A)$.
For each $i,j\in S$, let $A_{ij}=e_{ii}Ae_{jj}$. Let $\FF$ be the family of finite subsets of $S$, directed by inclusion.
For each $F\in \FF$ let
\[
A_F=\sum_{i,j\in F} A_{ij}.
\]
Fix $p\in S$. For all $i,j$ there is a linear bijection $\tau_{ij}:A_{ij}\to A_{pp}$ given by
\[
\tau_{ij}(a)=e_{pi}ae_{jp}.
\]
Note that $\tau_{ij}\inv(a)=e_{ip}ae_{pj}$. For each $F\in\FF$ there is a \cs-isomorphism
\[
\phi_F:A_F\variso A_{pp}\xt M_F
\]
given by
\[
\phi_F\left(\sum_{i,j\in F}a_{ij}\right)
=\sum_{i,j\in F}\bigl(\tau_{ij}(a_{ij})\xt u_{ij}\bigr).
\]
Define a \cs-morphism $\sigma:A_{pp}\to A$ by the strictly convergent sum
\[
\sigma(a)=\sum_{i\in S}\ad e_{ip}(a)=\sum_{i\in S}\tau_{ii}\inv(a).
\]
Then
\begin{align*}
    e_{ij}\sigma(a)
    &=\sum_{k\in S}e_{ij}e_{kp}ae_{pk}
    \\&=e_{ip}ae_{pj}
    \\&=\sum_{k\in S}e_{kp}ae_{pk}e_{ij}
    \\&=\sigma(a)e_{ij},
\end{align*}
so $\sigma(A_{pp})\subset C(A,\iota)$.

Note that for $F\subset E$ in $\FF$ we have canonical embeddings $A_F\hookrightarrow A_E$ and $M_F\hookrightarrow M_E$,
and we have an inductive-limit isomorphism (see \cite[Proposition~11.4.1]{kadrin2})
\[
\phi:=\varinjlim \phi_F:\varinjlim A_F
\variso \varinjlim (A_{pp}\xt M_F),
\]
and
\[
\varinjlim A_F=A
\midtext{and}
\varinjlim (A_{pp}\xt M_F)
=A_{pp}\xt \KK
\]
because $\varinjlim M_F=\KK$.
\end{proof}


\end{document}